\documentclass[
reqno, 
12pt]{amsart}
\usepackage{amssymb,amsmath, amsthm,latexsym, mathrsfs}

\usepackage{a4wide}

\usepackage{hyperref}

\theoremstyle{plain}

\newtheorem{lemma}{Lemma}[section]

\newtheorem{proposition}[lemma]{Proposition} 
\newtheorem{corollary}[lemma]{Corollary}

\newtheorem{remark}[lemma]{Remark}

\newtheorem{definition}[lemma]{Definition}

\parskip=\bigskipamount

\def\Z{\mathbb Z}
\def\R{\mathbb R}
\def\d{\delta}

\def\t{\times}
\def\o{\otimes}

\def\h{\hookrightarrow}
\def\ra{\rightarrow}

\def\a{\alpha}
\def\b{\beta}

\def\D{\Delta}
\def\G{\Gamma}

\def\s{\sigma}

\title[General linear groups over non-archimedean local division algebra]
{Remark on representation theory of general linear groups over a non-archimedean local division algebra}

\mark{Marko Tadi\'c}

\author{Marko Tadi\'c}

\address{Department of Mathematics, University of Zagreb
\\
Bijeni\v{c}ka 30, 10000 Zagreb,
 Croatia\\
Email: \tt tadic{\char'100}math.hr}


\thanks{This work has been  supported  by Croatian Science Foundation under the
project 9364.}
\date{\today}

\begin{document}

\begin{abstract} 
In this paper we give a  simple (local) proof of two  principal results about irreducible tempered representations of general linear groups over a non-archimedean local division algebra. We give a proof of the parameterization of the irreducible square integrable representations of these groups by segments of cuspidal representations, and a proof of the irreducibility of the tempered parabolic induction. Our proofs are based on Jacquet modules (and the Geometric Lemma, incorporated  in the structure of a Hopf algebra). We  use only some very basic general facts of the representation theory of reductive $p$-adic groups (the theory that we use was completed more then three decades ago, mainly in 1970-es). Of the specific results for general linear groups over $A$, basically we use only  a very old result of G.I.  Ol'\v sanski\v i,  which says that there exist complementary series starting from $Ind(\rho\o\rho)$ whenever $\rho$ is a unitary irreducible cuspidal representation. In appendix of \cite{L-M}, there is also a simple local proof of these results, based on a slightly different approach.
\end{abstract}

\maketitle

\setcounter{tocdepth}{1}


\section{Introduction}\label{intro}

Let $A$ be a local non-archimedean division algebra.
Two major steps in classifying unitary duals of general linear groups over a local non-archimedean division algebra $A$ modulo cuspidal representations, were done in \cite{BR} by D. Renard and I. Badulescu, and in \cite{Se} by V. S\'echerre.
The unitarizability of Speh representations was proved in \cite{BR} using global methods (and a simple form of the trace formula of Arthur\footnote{See 
\cite{Ba-JL} for more details regarding the trace formula used in the proof of the unitarizability in \cite{BR}.}). I. Badulescu has obtained in \cite{Ba-Speh} a local proof of this fact, simplifying substantially the previous proof of unitarizability of Speh representations.
The irreducibility of the unitary parabolic induction for groups $GL(n,A)$ was proved in \cite{Se} using the classification of simple types for these groups (and relaying on several other very powerful tools). In \cite{L-M}, E. Lapid and A. M\'inguez have recently obtained a proof of this fact based on Jacquet modules. Their proof is significantly simpler then the previous proof of the irreducibility of the unitary parabolic induction.

A consequence of  these two new proofs is that now we have  a very simple and rather elementary solution of the unitarizability problem for the groups $GL(n,A)$. These new developments  also simplify the classification in the case of the commutative $A$ obtained in \cite{T-AENS}. For example, the proof of  E. Lapid and A. M\'inguez  enables avoiding the use of a very important and delicate result of J. Bernstein from \cite{Be-P-inv} about restrictions of irreducible unitary representations of $GL(n)$ over a commutative $A$, to the mirabolic subgroup. Since this result does not hold for non-commutative $A$, E. Lapid and A. M\'inguez approach 
further provides a uniform approach to the commutative and non-commutative case (making no difference between them, which was the case before).

The  unitarizability that we have discussed above was solved using the simple basic results of the theory of non-unitary duals of these groups. 
 These simple basic results in characteristic zero for the setting of the Langlands classification in the case of non-commutative $A$ were obtained in \cite{T-Crelle}\footnote{Such results in the setting of the Zelevinsky classification were proved in the field case in \cite{Z}, using the theory of derivatives.}. The methods of the proofs in \cite{T-Crelle} are very simple. They relay  on two results of \cite{DKV}. The first is that the square integrable representations of these groups are parameterized by segments of cuspidal representations. The second one is that the tempered parabolic induction is irreducible. The proofs of these results in \cite{DKV} are based on global methods (and the simple trace formula of P. Deligne and D. Kazhdan; see \cite{DKV} for more details). I. Badulescu transferred these results to the positive characteristics in \cite{Ba-positive}. His proof is not very simple. It includes transfer of orbital integrals, trace Paley-Wiener theorem of I. Bernstein, P. Deligne and D. Kazhdan, and it uses the close fields to relate the problem to the characteristic zero. Therefore, again in the background are global methods and the simple trace formula. A. M\'inguez and V. S\'echerre provided in \cite{M-S-1} a local proof of results of \cite{T-Crelle}, relaying at some point on simple types. Their paper directs much bigger generality. It address ("banal") modular representations, which include as special case the representations over the complex field. Their paper also covers additionally the case of Zelevinsky classification.

In this paper we give a proof of the parameterization of irreducible square integrable representations by segments of cuspidal representations, and a proof of the irreducibility of the tempered parabolic induction. Our paper uses the general facts of the representation theory of reductive $p$-adic groups (the theory that we use was completed more then three decades ago, mainly in 1970-es). Of the specific results for general linear groups over $A$, basically we use only  a very old result of G.I.  Ol'\v sanski\v i from \cite{O}\footnote{As E. Lapid indicated to us, this result is reproved in \cite{Sh+} in the field case (with a little bit of additional work; see Proposition in \cite{Sh+}). His proof naturally extends also to the division algebra case.},  which says that there exist complementary series starting from $Ind(\rho\o\rho)$ whenever $\rho$ is a unitary irreducible cuspidal representation. Then the positive reducibility point determines the character $\nu_\rho$, which is used to define the segments\footnote{A. M\'inguez and V. S\'echerre in their paper \cite{M-S-1} use simple types from \cite{M-S-2} instead of the G.I.  Ol'\v sanski\v i  result. The methods (and aims) of the paper of A. M\'inguez and V. S\'echerre and ours are very different (in our paper the unitarity is crucial, while in their paper the unitarity does not play a role at all).}. Using the Jacquet modules, in this paper we also reprove basic results in \cite{T-Crelle} (the proofs of \cite{T-Crelle} are more based on the properties of the Langlands classification of the non-unitary dual).
The first fact that we prove in this paper is the  irreducibility of the representation parabolically induced by a tensor product of several  square integrable representations corresponding to segments.
The second fact that we prove is the  irreducibility of the representation parabolically induced by a tensor product of two essentially square integrable representations corresponding to segments which are not linked (i.e. that their union is not a segment which is different from both of the segments). Now,  for getting the irreducibility of the tempered parabolic induction,  it is enough to prove that all irreducible square integrable representations are coming from segments. The proof of this fact was done by C. Jantzen in \cite{J} in the case of commutative $A$ (using Jacquet modules). Since his proof uses some facts of the Bernstein-Zelevinsky theory, we slightly modify his argument to apply to the case of general division algebra (since here we cannot relay on the Bernstein-Zelevinsky theory).

Together with Theorem 4.3   of \cite{Sh}\footnote{Basically, this theorem tells that $Ind(\rho\o|\ |_F\rho)$ reduces for an irreducible cuspidal representation $\rho$, in the case that $A$ is commutative.}, 
our paper provides also an alternative approach to some of the most basic results of the Bernstein-Zelevinsky theory
  in the setting of the Langlands classification 
  (here $A$ is commutative; see \cite{Ro} for more details).
  
   In the appendix of \cite{L-M} there is a simple local proof not using types
  of two  results of \cite{DKV} which we have mentioned above, on which
\cite{T-Crelle} relays.
 In this paper we give also such a proof of that two principal results. The proofs in the appendix of \cite{L-M} and in this paper have some parts in common, but the general strategy of the proofs is different (in our paper the  square integrability plays the central role). Another difference is that proofs in the appendix of \cite{L-M} are much more concise then the proofs in this paper (the appendix of \cite{L-M} covers more general setting). This paper is written in more elementary way, and it is more self-contained (which contributes to the length of the paper) and we consider that it may be of some interest to have also this approach available. The paper \cite{L-M} is very important, and it opens new possibilities  to attack some fundamental problems of the representation theory of $p$-adic general linear groups.

The first version of this simple paper was finished in the spring of 2014,  when the author was guest of the  Hong Kong University of Science and Technology, and we are thankful  for their kind hospitality. Lengthy discussions with E. Lapid helped us to make clear some basic topics used in this paper.
I. Badulescu suggested us a  number of corrections and  modifications, which make the paper easier to read. Discussions with A. M\'inguez, A. Moy and G. Mui\'c during writing of this pretty elementary paper were useful to us. We are thankful to all of them, and to
 E. Lapid, C. M\oe glin, G. Mui\'c,
F. Shahidi and J.-L. Waldspurger for  discussions related to the question of uniqueness (up to a sign) of the reducibility point on the "real axes", when one induces from a maximal parabolic subgroup by an irreducible cuspidal representation.

The content of the paper is the following. The second section introduces notation used in the paper and recalls of some very basic facts that we shall use later. The irreducibility of a product of irreducible square integrable representation corresponding to segments of cuspidal representations is proved in the third section.
In the fourth section of this paper we prove the irreducibility of representation parabolically induced by a tensor product of two essentially square integrable representations corresponding to segments which are not linked. In the fifth section we collect some very elementary facts about Jacquet modules that we need for the sixth section, where we recall of the Jantzen's completeness argument for irreducible square integrable representations. In the seventh section we present determination of the composition series, based on the Jacquet module methods, of the reducible product of two irreducible essentially square integrable representations   (these composition series were determined in \cite{T-Crelle} using the properties of the Langlands classification).

\section{Notation and preliminary results}\label{notations}

For a reductive $p$-adic group $G$, the Grothendieck group $R(G)$ of the category $Alg_{f.l.}(G)$ of smooth representations of finite length has natural ordering $\leq$ (the cone of positive elements is generated by irreducible representations). There is mapping from $Alg_{f.l.}(G)$ into $R(G)$ called semi simplification, and denoted by $s.s.$. To simplify notation, we shall write below the fact 
$s.s.(\pi_1)\leq s.s.(\pi_2)$ simply as $\pi_1\leq \pi_2$.

Fix a non-archimedean local algebra $A$. 
We shall consider in this paper admissible representations of groups $G_n:=GL(n,A)$ of finite length. We use in this paper the notation of \cite{T-Crelle} (most of this notation is an extension of the notation of J. Bernstein and A. V. Zelevinsky which they were using in their earlier papers). We shall recall very briefly of some notation that we shall use very often. We use the notation 
$
\t
$
 for the parabolic induction from maximal standard\footnote{Standard parabolic subgroups are those ones containing the regular upper triangular matrices. All the Jacquet modules that we shall consider in this paper will be with respect to standard maximal parabolic subgroups.} parabolic subgroups.
We shall use the  graded positive Hopf algebra $R$ introduced  in the  third section of \cite{T-Crelle} (introduced in the same way as it was introduced in \cite{Z} in the field case). The multiplication in $R$ is defined using the parabolic induction, while the comulitplication $m^*$ is defined using Jacquet modules with respect to maximal standard  parabolic subgroups. The algebra is commutative (but not cocommutative). All the Jacquet modules that we shall need in this paper, we shall compute using this Hopf algebra.

We identify characters of $A^\t$ with that of $F^\t$ in usual way (using non-commutative determinant). Then for two admissible representations $\pi_1$ and $\pi_2$ holds
\begin{equation}
\label{chi}
\chi(\pi_1\t\pi_2)\cong (\chi\pi_1)\t(\chi \pi_2).
\end{equation}
Denote
$$
\nu:=| \det |_F,
$$
where $\det$ denotes the determinate homomorphism, defined by Dieudonn\'e.

Let $\rho_1$ and $\rho_2$ be  irreducible cuspidal representations of general linear groups over a fixed non-archimedean local algebra $A$. 
Suppose that $\rho_1\t\rho_2$ reduces.
Suppose that the central character of  $\rho_1\t\rho_2$ is unitary.
 We shall now consider two possibilities. 
 
 The first is when $\rho_1$ is unitary (then $\rho_2$ is also unitary). Now Frobenius reciprocity implies that 
 $\rho_1\o\rho_2$ has multiplicity (at least) two in the Jacquet module of  $\rho_1\t\rho_2$. Since this representation has precisely two irreducible subquotients in the Jacquet modules which are cuspidal, and they are  $\rho_1\o\rho_2$ and  $\rho_2\o\rho_1$, this implies 
$$
 \rho_1\cong\rho_2.
$$

Suppose now that $\rho_2$ is not unitary. Then Casselman square integrability criterion (Theorem 6.5.1 of \cite{Ca}) implies that  $\rho_1\t\rho_2$ has an irreducible square integrable subquotient. This implies that  $\rho_1\t\rho_2$ and  $\rho_1^+\t\rho_2^+$ have an irreducible subquotient in common, which further implies that  $\rho_1\cong\rho_1^+$ or  $\rho_1\cong\rho_2^+$ ($\pi^+$ denotes the Hermitian contragredient of $\pi$). The first possibility is excluded, since $\rho_1$ is not unitary. Thus, the second possibility must hold.
We can write
$$
\rho_1=
\nu^{s/2}\rho,
$$
with $\rho$ unitary and $s\in\mathbb R$, $s\ne 0$. Interchanging $\rho_1$ and $\rho_2$ if necessary, we can assume $s> 0$. 
Observe that now the above condition implies
$$
\rho_2\cong \nu^{-s/2}\rho.
$$
From this (using \eqref{chi}) follows that
$$
\rho
\t
\nu^s\rho
$$
reduces.

It is well known (and easy to prove) that there exists some $\epsilon>0$ such that $\nu^t\rho\t\nu^{-t}\rho$ is irreducible for all $0<t<\epsilon$. Further, Theorem 2 of \cite{O} implies that one can chose $0<\epsilon'\leq \epsilon$ such that $\nu^t\rho\t\nu^{-t}\rho$ is unitarizable for all $0<t<\epsilon'$ (it is also irreducible for that $t$'s). In other words, we have complementary series here.
Since complementary series cannot go to infinity (see \cite{T-Geom}), we must have reducibility of $\nu^t\rho\t\nu^{-t}\rho$ for some $t>  0$. Chose smallest such $t\geq 0 $ such that $\nu^t\rho\t\nu^{-t}\rho$ reduces, and denote it by $\frac s2$.

Recall that    Lemma 1.2  of \cite{Si-2}\footnote{See also Theorem 1.6 there, and i) of Proposition 4.1 of \cite{He}.} implies that above $s$ is unique. It is determined by $\rho$. Therefore, we shall denote it by
$$
s_\rho.
$$
Observe that now the uniqueness of $s_\rho\geq 0$ which gives reducibility, implies
\footnote{ Recall that \cite{DKV} gives a  precise description  of $s_\rho$ in terms of the Jacquet-Langlands correspondence
 (among others,  it implies that it is an integer which divides the rank of $A$). 
 The theory of types gives a different description  of $s_\rho$ (see section 4 of \cite{Se}). We shall not use any of such additional information in this paper. 
In this manuscript we shall relay neither on (simple) trace formula (and global methods), nor on the theory of types.}
$$
s_\rho>0
.
$$

We denote
$$
\nu_\rho=
\nu^{s_\rho}.
$$
The relation \eqref{chi} implies $s_\rho=s_{\chi\rho}$ and further $\nu_\rho=\nu_{\chi\rho}$.

The set of the form
$$
\{\rho, \nu_\rho\rho,\dots,\nu_\rho^k\rho\}
$$
will be called a segment (in cuspidal representations), and it will be denoted by
$$
[\rho,\nu_\rho^k\rho].
$$

The representation 
$$
\nu_\rho^k\rho\t\nu_\rho^{k-1}\rho\t\dots \t\rho,
$$
contains  a unique irreducible subrepresentation,
which is denoted by
$$
\d([\rho,\nu_\rho^k\rho])
$$
(uniqueness of the irreducible subrepresentation follows from the regularity, i.e.  from the fact that all the Jacquet modules of the induced representation are multiplicity free).

Observe that for $0\leq i\leq k-1$ we have
$$
\nu_\rho^k\rho\t\nu_\rho^{k-1}\rho\t\dots \t\nu_\rho^{i+2}\rho\t\d([\nu_\rho^i\rho,\nu_\rho^{i+1}\rho])\t\nu_\rho^{i-1}\rho\t\dots\t\nu_\rho\rho \t\rho
\h
\nu_\rho^k\rho\t\nu_\rho^{k-1}\rho\t\dots \t\rho.
$$
From this and the definition of 
 $\d([\rho,\nu_\rho^k\rho  ])$ follows
$$
 \d([\rho,\nu_\rho^k\rho  ])\h 
\nu_\rho^k\rho\t\nu_\rho^{k-1}\rho\t\dots \t\nu_\rho^{i+2}\rho\t\d([\nu_\rho^i\rho,\nu_\rho^{i+1}\rho])\t\nu_\rho^{i-1}\rho\t\dots\t\nu_\rho\rho \t\rho.
$$
Applying induction, this implies that the minimal non-trivial (standard) Jacquet module of $ \d([\rho,\nu_\rho^k\rho  ])$ is 
$$
\nu_\rho^k\rho\o\nu_\rho^{k-1}\rho\o\dots \o\rho.
$$
Further, if $\pi$ is an irreducible representation which has the above representation in its Jacquet module, then $\pi \cong 
 \d([\rho,\nu_\rho^k\rho  ])$. This characterization of representations  $ \d([\rho,\nu_\rho^k\rho  ])$ directly implies the following formula\footnote{In \cite{T-Crelle}, we have obtained this formula in the same way  (from (iii) of Proposition 3.1 and (i) of Proposition 2.7 there).}
\begin{equation}
\label{m^*}
m^*(\d([\rho,\nu_\rho^k\rho]))=\sum_{i=-1}^{k}
\d([\nu_\rho^{i+1}\rho,\nu_\rho^k\rho])\o
\d([\rho,\nu_\rho^{i}\rho]).
\end{equation}

The square integrability criterion of Casselman\footnote{See \eqref{=} and \eqref{>} in section \ref{square}.} now implies that this representation is essentially square integrable.


For a segment $\D$ of cuspidal representations one can easily prove   that
$$
\d(\D)\tilde{\ }\cong \d(\tilde\D),
$$
where $\tilde\D=\{\tilde\rho;\rho\in \D\}.$
Further, from \eqref{chi} follows
$$
\chi\d(\D)\cong\d(\chi\D),
$$
where $\chi\D=\{\chi\rho;\rho\in\D\}.$ 

 We shall denote
 $$
 \D^+=\{\rho^+:\rho\in \D\},
 $$
  where $\rho^+$ denotes the Hermitian contragredient of $\rho$. We shall say that $\D$ is unitary if $\D=\D^+$. In this case $\d(\D)$ is unitarizable (since it is square integrable modulo center).


Let $\s=\d([\rho,\nu_\rho^k\rho])$. Then we shall denote 
$
s_\rho
$
also by $s_\s$,
and $\nu_\rho$ by $\nu_\s$.

Let $\pi$ and $\s$ be a representations of finite length of $G_n$ and $G_m$ respectively, where $n\geq m$. Suppose that $\s$ is irreducible. We denote by
$$
m^*_{-\o\sigma}(\pi)
$$
the sum of all irreducible terms in $m^*(\pi)$ which are of the form $\tau\o\s$ (counted with multiplicities). Analogously we define $m^*_{\sigma\o-}(\pi)$.

\section{Irreducibility - the case of several unitary segments}\label{irr-unitary-segments}


First we have a very well known\footnote{It is a special case of a result that holds for a general reductive group. For completeness (and since it is very simple), we present the proof for general linear groups here.} 

\begin{lemma}
\label{regular}
Let $\D_1,\D_2,\dots ,\D_n$ be different unitary segments of cuspidal representations. Then the representation
$$
\d(\D_1)\t\d(\D_2)\t\dots \t\d(\D_n)
$$
is irreducible.
\end{lemma}

\begin{proof} Since segments $\D_i$ are different and unitary, we can enumerate them in a way that 
\begin{equation}
\label{indexing}
\text{for   $1\leq j<n$, 
both ends of  $\D_j$ are not contained in   $\D_k$, for any $j<k\leq n$.}
\end{equation}
 The commutativity of algebra $R$ implies that it is enough to prove the lemma for such enumerated segments.

Now the formula for $m^*(\d(\D_1))$ implies that
$$
m^*_{\d(\D_1)\o-}(\d(\D_1)\t \d(\D_2)\t\dots \t\d(\D_n))= \d(\D_1)\o \d(\D_2)\t\dots \t\d(\D_n).
$$
Continuing this procedure with $\D_2,\dots ,\D_n$, etc. (in each step delating the segment with the lowest index) and using the transitivity of Jacquet modules, we get that the multiplicity of $\d(\D_1)\o \d(\D_2)\o\dots \o\d(\D_n)$ in the Jacquet module of $\d(\D_1)\t \d(\D_2)\t\dots \t\d(\D_n)$ is one. Now the unitarizability of $\d(\D_1)\t \d(\D_2)\t\dots \t\d(\D_n)$ and the Frobenius reciprocity imply the irreducibility.
\end{proof}

Let an irreducible representation $\s$ of some $G_m$ be a subquotient of $\rho_1\dots\t\rho_l$, where $\rho_i$ are irreducible and cuspidal. Then the multiset $(\rho_1,\dots,\rho_l)$  is called the cuspidal support of $\s$. It is denoted by 
$$
\text{supp$(\s)$.}
$$
For each  $1\leq i\leq n$ fix a  finite multiset $X_i$  of irreducible cuspidal representations of general linear groups over $A$. Let $\pi$ be a representation of finite length of some $G_m$. We denote by
$$
m^*_{\text{supp}(X_1,\dots,X_n)}(\pi)
$$
the sum\footnote{The sum is in the Grothendieck group of the category of finite-length representations of the corresponding Levi subgroup.} (counted with multiplicities) of all irreducible terms in the corresponding Jacquet module of $\pi$ which are of the form $\tau_1\o\dots\o\tau_n$  such that supp$(\tau_i)=X_i$ for all $1\leq i\leq n$.
\begin{lemma}
\label{same}
Let $\D$ be a unitary segment of cuspidal representations. Then 
$$
\underbrace{\d(\D)\t\dots\t\d(\D)}_{k-\text{times}}
$$
is irreducible\footnote{Twisting by a character, one directly sees that the claim of the lemma holds without assumption that the segment is unitary.}.
\end{lemma}

\begin{proof}  Denote the representation whose irreducibility we want to prove in the above lemma by $\pi$. Suppose that it is not irreducible. Then we can write
$$
\pi=\pi_1\oplus\pi_2,
$$
where $\pi_i$ are non-zero subrepresentations.

Write
$$
\D=[\nu^a_\rho\rho,\nu^{-a}_\rho\rho].
$$
Introduce segments
$$
\G_i=[\nu^{a+1}_\rho\rho,\nu^{a+i}_\rho\rho], \quad i=1,\dots,k,
$$
$$
\D_i=\G_i^+\cup \D\cup \G_i=[\nu^{-a-i}_\rho\rho,\nu^{-a+i}_\rho\rho], \quad i=1,\dots,k.
$$
We can consider segments as multisets, and introduce multisets
$$
X_-=\sum_{i=1}^k \G_i^+,  \qquad X_+=\sum_{i=1}^k \G_i, \qquad X=\sum_{i=1}^k\D.
$$

Suppose that $\D'$ and $\D''$ are disjoint segment of cuspidal representations, such that their union is again  a segment of cuspidal representations. Then one proves easily that $\d(\D'\cup\D'')$ is a subquotient of  $\d(\D')\t\d(\D'')$ (for a proof of this very elementary fact see the very beginning of the proof of Lemma \ref{l1}).
This (together with the commutativity of $R$) directly implies  
$$
\d(\D_1)\t\dots\t\d(\D_k)\leq \d(\G_1)\t\dots\t\d(\G_k)\t\pi\t \d(\G_1^+)\t\dots\t\d(\G_k^+).
$$
Since the representation on the left hand side is irreducible by the previous lemma, we get that 
$$
\d(\D_1)\t\dots\t\d(\D_k)\leq \d(\G_1)\t\dots\t\d(\G_k)\t\pi_i\t \d(\G_1^+)\t\dots\t\d(\G_k^+),
$$
for at least one $i\in\{1,2\}$. Fix some  $i$ satisfying this. Then obviously
$$
m^*_{\text{supp}(X_+,X,X_-)}(\d(\D_1)\t\dots\t\d(\D_k))
\leq 
\hskip60mm
$$
$$
\hskip40mm
m^*_{\text{supp}(X_+,X,X_-)}(\d(\G_1)\t\dots\t\d(\G_k)\t\pi_i\t \d(\G_1^+)\t\dots\t\d(\G_k^+)).
$$
Denote
$$
\Lambda= \d(\G_1)\t\dots\t\d(\G_k)\o\pi\o \d(\G_1^+)\t\dots\t\d(\G_k^+),
$$
$$
\Lambda_i= \d(\G_1)\t\dots\t\d(\G_k)\o\pi_i\o \d(\G_1^+)\t\dots\t\d(\G_k^+).
$$
Obviously $\Lambda \not\leq \Lambda_i$ (consider the lengths of both sides).

Using the formula \eqref{m^*}, one directly checks that
$$
\Lambda\leq m^*_{\text{supp}(X_+,X,X_-)}(\d(\D_1)\t\dots\t\d(\D_k))\footnote{Actually, we have here equality.}
$$
and
$$
m^*_{\text{supp}(X_+,X,X_-)}\d(\G_1)\t\dots\t\d(\G_k)\t\pi_i\t \d(\G_1^+)\t\dots\t\d(\G_k^+))=\Lambda_i.
$$
Now the last inequality above implies $\Lambda\leq \Lambda_i$. This contradict to our previous observation that $\Lambda \not\leq \Lambda_i$. The proof of the lemma is now complete.
\end{proof}

\begin{proposition}
\label{irr-uni-seg}
Let $\G_1,\G_2,\dots ,\G_m$ be  unitary segments of cuspidal representations. Then the representation
$$
\d(\G_1)\t\d(\G_2)\t\dots \t\d(\G_m)
$$
is irreducible.
\end{proposition}

\begin{proof} We can write the multiset $(\G_1,\G_2,\dots ,\G_m)$ as
$$
(\underbrace{\D_1,\dots,\D_1}_{k_1-\text{times}}, \underbrace{\D_2,\dots,\D_2}_{k_2-\text{times}}, \dots, \underbrace{\D_n,\dots,\D_n}_{k_n-\text{times}}),
$$
where $\D_1,\dots,\D_n$ are different segments satisfying \eqref{indexing}. Denote
$$
\pi_i=\underbrace{\d(\D_i)\t\dots\t\d(\D_i)}_{k_i-\text{times}}.
$$
Then the above lemma tells us that these representations are irreducible. For the proof of the proposition, it is enough to show that $\pi_1\t\dots\t\pi_n$ is irreducible.

In the same way as in the proof of Lemma \ref{regular} one gets that
%
%
$$
m^*_{\pi_1\o-}(\pi_1\t\pi_2\t\dots\t\pi_n)= \pi_1\o\pi_2\t\dots\t\pi_n.
$$
Continuing this procedure (similarly as in the proof of Lemma \ref{regular}  and using the transitivity of Jacquet modules), we get that the multiplicity of $\pi_1\o\pi_2\o \dots \o \pi_n$ in the Jacquet module of $\pi_1\t\pi_2\t \dots \t \pi_n$ is one. Again the unitarizability of $\pi_1\t\pi_2\t \dots \t \pi_n$ and the Frobenius reciprocity give the irreducibility.
\end{proof}

\begin{remark} 
 Let $\rho$ be an irreducible cuspidal  representation of some $G_n$. 
  In Appendix of \cite{L-M} (Theorem A.1), there is a simple proof  of uniqueness of the reducibility point $s_\rho$ based on Jacquet modules (without using a non-elementary analytic  argument from  \cite{Si-2}). 
  We shall briefly describe idea of that proof (the proof in \cite{L-M} is concise; our brief description of a special case of  that proof is longer than the general proof in \cite{L-M}).
  
  Observe that if we know that there is a reducibility point $s_\rho$ which is strictly positive (for which we do not need to know that it is unique), one  defines $\nu_\rho$ using that $s_\rho$, and then segments of cuspidal representations $\{\rho, \nu_\rho\rho,\dots,\nu_\rho^k\rho\}$. Further, one attaches to such segments essentially
   square integrable representation as before. Now Lemma \ref{same} implies that if there is strictly positive reducibility, $\rho\t\rho$ must be irreducible. Therefore, to prove the uniqueness of $s_\rho$, it is enough to show that one can not have more then one strictly positive reducibility point. 
We shall sketch  how to prove that one can not have reducibility of both $\nu_\rho\rho\t \rho$ and $\nu_\rho^a\rho\t\rho$ for some $a>1$ (the case $a<1$ reduces to this case interchanging the reducibility point in the definition of $\nu_\rho$). To simplify technicalities, we shall give the  argument from \cite{L-M} in the case $a\leq2$ (the idea in general case is the same).

Suppose that both $\nu_\rho\rho\t \rho$ and $\nu_\rho^a\rho\t\rho$ are reducible, where $1<a\leq 2$. 
Then $\nu_\rho\rho\t \rho$ (resp. $\nu_\rho^a\rho\t\rho$) has unique irreducible subrepresentation, and it is essentially square integrable. We denote it by $\d_1$ (resp. $\d_a$). The minimal non-zero (standard) Jacquet module of this subrepresentation is $\nu_\rho\rho\o \rho$ (resp. $\nu_\rho^a\rho\o\rho$).
 Twisting $\rho$ by a character, we can reduce the proof to the case  when $\nu_\rho^a\rho\t\d_1$ has unitary central character (use \eqref{chi}). Observe that $\nu^a_\rho\rho\t\d_1\h
\nu_\rho^a\rho\t\nu_\rho\rho\t\rho$, and that the right hand representation has a unique irreducible subrepresentation (since the induced representation is regular). Denote this irreducible subrepresentation by $\s$. Obviously $\s\leq  \nu^a_\rho\rho\t\d_1$. A simple consideration of Jacquet modules and Frobenius reciprocity imply that there exists a non-trivial intertwining $\d_a\t \nu_\rho\rho\ra \nu_\rho^a\rho\t\nu_\rho\rho\t\rho$ (we are in the regular situation). Thus $\s\leq \d_a\t\nu_\rho \rho$. 

Consider first the case $a=2$. Then  $\d_2$ and $\nu_\rho \rho$ are square integrable (since they have unitary central characters). Therefore, $\d_2\t\nu_\rho \rho$ has no square integrable subrepresentations. From the other side,   $\s=\d([\rho,\nu_\rho^2\rho])\leq \d_2\t\nu_\rho \rho$, which  is square integrable. This contradiction ends the sketch of the proof in this case. 

Consider now the remaining case $a<2$. Denote the minimal non-zero (standard) Jacquet module of $\s$ by $\bold r(\s)$, and denote the Jacquet module of $\s$ with respect to the opposite parabolic subgroup by $\bar{\bold r}(\s)$. The above estimates $\s\leq  \nu^a_\rho\rho\t\d_1$and $\s\leq \d_2\t\nu_\rho \rho$ imply easily\footnote{This upper estimate has more terms in the case $a>2$.}  $\bold r(\s)\leq \nu^a_\rho\rho\o \nu_\rho\rho\o\rho+\nu_\rho\rho\o \nu_\rho^a\rho\o\rho$. Now the Casselman square integrability criterion (see \eqref{=} and \eqref{>}) implies that $\s$ is square integrable, which further implies $\s\cong \s^+$. Now we shall use the fact that  the Jacquet module (with respect to a standard parabolic subgroup) of the  contragredient  representation, is the contragredient of the Jacquet module with respect to the opposite parabolic subgroup of the representation (see Corollary 4.2.5 of \cite{Ca} for precise statement). 
This fact and $\s\cong \s^+$ imply ${\bold r}(\s)\cong {\bold r}(\s^+)\cong (\bar{\bold r}(\s))^+$.

One gets $\bar{\bold r}(\s)$ conjugating ${\bold r}(\s)$ by appropriate element of the Weyl group (which conjugates the standard parabolic subgroup from which $\nu_\rho^a\rho\t\nu_\rho\rho\t\rho$ is induced,  to the opposite one). This and the above estimate of ${\bold r}(\s)$ give $\bar{\bold r}(\s)\leq \rho\o\nu_\rho\rho\o \nu_\rho^a\rho+\rho\o\nu^a_\rho\rho\o \nu_\rho\rho$. Therefore, $(\bar{\bold r}(\s))^+\leq \rho^+\o\nu_\rho^{-1}\rho^+\o \nu_\rho^{-a}\rho^+ +\rho^+\o\nu^{-a}_\rho\rho^+\o \nu_\rho^{-1}\rho^+$.
The fact  ${\bold r}(\s)\cong  (\bar{\bold r}(\s))^+$ gives a new upper bound for ${\bold r}(\s)$ (the same as for $  (\bar{\bold r}(\s))^+$). These two upper bounds of ${\bold r}(\s)$, and the condition  that the central character of  $\nu_\rho^a\rho\t\d_1$ is unitary,  directly imply ${\bold r}(\s)=0$. This contradiction completes the sketch of the proof.
\end{remark}

\section{Irreducibility - the case of two non-linked segments}\label{irr}

\begin{definition}
For a representation $\pi$ of a finite length we denote
$$
m^*_{bottom}(\pi)=\sum
\pi'\o\rho',
$$
where the sum runs over all irreducible $\pi'\o\rho'$ in $m^*(\pi)$ (counted with multiplicities) such that $\rho'$ is cuspidal.
\end{definition}

\begin{lemma}
\label{subset}
Let $\D_1\subseteq 
\D_2$ be two non-empty segments as above. Then
$$
\d(\D_1)\t\d(\D_2)
$$
is irreducible.
\end{lemma}
 
\begin{proof}
 Twisting with a character, we can reduce the proof of the lemma to the case when $\d(\D_1)$ is unitary.
We shall assume this is the rest of the proof.

First observe that Proposition \ref{irr-uni-seg} implies that  the claim of the above lemma holds in the case that  $\d(\D_2)$ is also unitary.


We shall prove the remaining cases of the lemma by induction with respect to  the sum of lengths of $\D_1$ and $\D_2$.  If the sum of the lengths is 2, then we know that the lemma holds. Therefore, we shall fix two segments whose sum of lengths is at least three, and suppose that the lemma holds for all the pairs of segments whose sum of lengths is strictly smaller then the sum of the lengths of the  segments that we have fixed. 


It is enough to consider the case $\D_1\ne\D_2$. We can write
$$
\D_1=[\nu_\rho^{-n}\rho, \nu_\rho^{n}\rho], \ \ n\in(1/2)\Z, \ \ n
\geq 0,
$$
$$
\D_2=[\nu_\rho^{-n_-}\rho, \nu_\rho^{n_+}\rho], \quad n_-,n_+\in(1/2)\Z, \ \  n_--n_+\in \Z, 
$$
where
$$
 \ \ n-n_-\in \Z,  \ \ n_-,n_+
\geq n.
$$
Since we have seen that the lemma holds in the case $n_-=n_+$, it remains to prove the lemma when $n_+<n_-$ or $n_+>n_-$ We shall now prove the lemma in the case $n_+<n_-$. 
The other case  (i.e. $n_-<n_+$) follows in a similar  way, or one can get it also from the  case $n_+<n_-$ applying the (hermitian) contragredient.

Assume $n_+<n_-$.
Then $n\leq n_+$ implies $n<n_-$.
Suppose that $\d(\D_1)\t\d(\D_2)$ reduces. Observe that
$$
m^*_{bottom}(\d(\D_1)\t\d(\D_2))=
$$
$$
\d([\nu_\rho^{-n}\rho, \nu_\rho^{n}\rho])
\t \d([\nu_\rho^{-n_-+1}\rho, \nu_\rho^{n_+}\rho])\o \nu_\rho^{-n_-}\rho
+
\d([\nu_\rho^{-n+1}\rho, \nu_\rho^{n}\rho])\t \d([\nu_\rho^{-n_-}\rho, \nu_\rho^{n_+}\rho]) \o \nu_\rho^{-n}\rho.
$$
Observe that both representations in this Jacquet module are irreducible by the inductive assumption. Therefore, $\d(\D_1)\t\d(\D_2)$ is a length two representation. Write it in the Grothendieck group as a sum of irreducible representations
\begin{equation}
\label{sum}
\d(\D_1)\t\d(\D_2)=\pi_1+\pi_2.
\end{equation}
Then with a suitable choice of indexes we have 
\begin{equation}
\label{B1}
m^*_{bottom}(\pi_1)=\d([\nu_\rho^{-n}\rho, \nu_\rho^{n}\rho])
\t \d([\nu_\rho^{-n_-+1}\rho, \nu_\rho^{n_+}\rho])\o \nu_\rho^{-n_-}\rho,
\end{equation}
\begin{equation}
\label{B2}
m^*_{bottom}(\pi_2)=\d([\nu_\rho^{-n+1}\rho, \nu_\rho^{n}\rho])\t \d([\nu_\rho^{-n_-}\rho, \nu_\rho^{n_+}\rho]) \o \nu_\rho^{-n}\rho.
\end{equation}
Consider now the representation
$$
\pi:=\d([\nu_\rho^{-n}\rho, \nu_\rho^{n}\rho])
\t \d([\nu_\rho^{-n_-}\rho, \nu_\rho^{n_-}\rho]).
$$
We know by the first part of the proof that this representation is irreducible (since $n\ne n_-$).
Observe that the formula \eqref{m^*} implies
$$
\d([\nu_\rho^{-n_-}\rho, \nu_\rho^{n_-}\rho])
\h \d([\nu_\rho^{n_++1}\rho, \nu_\rho^{n_-}\rho])\t \d([\nu_\rho^{-n_-}\rho, \nu_\rho^{n_+}\rho]).
$$
This implies
$
\pi\h
$
$$
\d([\nu_\rho^{-n}\rho, \nu_\rho^{n}\rho])\t
\d([\nu_\rho^{n_++1}\rho, \nu_\rho^{n_-}\rho])\t \d([\nu_\rho^{-n_-}\rho, \nu_\rho^{n_+}\rho])
=\d(\D_1)\t\d([\nu_\rho^{n_++1}\rho,\nu_\rho^{n_-}\rho])\t\d(\D_2).
$$
This implies that in $R$ we have
$$
\pi \leq \d(\D_1)\t\d(\D_2) \t \d([\nu_\rho^{n_++1}\rho,\nu_\rho^{n_-}\rho]).
$$
Since $\pi$ is irreducible, \eqref{sum} implies  
$$
\pi\leq \pi_1\t \d([\nu_\rho^{n_++1}\rho, \nu_\rho^{n_-}\rho])
$$
or
$$
\pi\leq \pi_2\t \d([\nu_\rho^{n_++1}\rho, \nu_\rho^{n_-}\rho]).
$$
We shall now show that neither of these two possibilities can happen, which will complete the proof.

Suppose that the first inequality holds, 
 i.e. $ \pi\leq \pi_1\t \d([\nu_\rho^{n_++1}\rho, \nu_\rho^{n_-}\rho])$. 
Now \eqref{B1} 
implies 
$$
m^*_{bottom}(\pi_1\t \d([\nu_\rho^{n_++1}\rho, \nu_\rho^{n_-}\rho]))=
\pi_1\t \d([\nu_\rho^{n_++2}\rho, \nu_\rho^{n_-}\rho])\o \nu_\rho^{n_++1}\rho
$$
$$
+\ \d([\nu_\rho^{-n}\rho, \nu_\rho^{n}\rho])
\t \d([\nu_\rho^{-n_-+1}\rho, \nu_\rho^{n_+}\rho])
\t \d([\nu_\rho^{n_++1}\rho, \nu_\rho^{n_-}\rho])\o \nu_\rho^{-n_-}\rho.
$$

Obviously, in the Jacquet module of $\pi_1\t \d([\nu_\rho^{n_++1}\rho, \nu_\rho^{n_-}\rho])$ 
we shall never have the term which finishes with $\dots\o\nu_\rho^{-n}\rho$. Since $\pi$ has at least one such term in its Jacquet module, we have got a contradiction. Therefore, the first inequality can not happen.

Therefore,  the second inequality holds, i.e. $ \pi\leq \pi_2\t \d([\nu_\rho^{n_++1}\rho, \nu_\rho^{n_-}\rho])$. 
Now \eqref{B2} implies
$$
m^*_{bottom}(\pi_2 \t \d([\nu_\rho^{n_++1}\rho, \nu_\rho^{n_-}\rho]))=
\pi_2\t \d([\nu_\rho^{n_++2}\rho, \nu_\rho^{n_-}\rho])\o \nu_\rho^{n_++1}\rho 
$$
$$
+\  \d([\nu_\rho^{-n+1}\rho, \nu_\rho^{n}\rho])\t \d([\nu_\rho^{-n_-}\rho, \nu_\rho^{n_+}\rho]) \t \d([\nu_\rho^{n_++1}\rho, \nu_\rho^{n_-}\rho])\o \nu_\rho^{-n}\rho.
$$

Obviously, in the Jacquet module of $\pi_2\t \d([\nu_\rho^{n_++1}\rho, \nu_\rho^{n_-}\rho])$ 
we shall never have the term which finishes with $\dots\o\nu_\rho^{-n_-}\rho$. Since $\pi$ has at least one such term in its Jacquet module, we have got again a contradiction.
This completes the proof.
\end{proof}

We say two segments $\D_1$ and $\D_2$ are linked, if $\D_1\cup \D_2$ is a segment which is different from both $\D_1$ and $\D_2$. For a segment $\D$, we shall denote by 
$$
b(\D)
$$
 its starting representation. We further denote 
$$
^-\D=\D\backslash b(\D).
$$

\begin{proposition}
\label{link}
Suppose that (non-empty) segments $\D_1$ and $\D_2$ are not linked. Then
$$
\d(\D_1)\t\d(\D_2)
$$
is irreducible.
\end{proposition}

\begin{proof}
Thanks to the previous lemma, it is enough to prove the proposition in the case $
\D_1\cap\D_2=\emptyset$. We shall assume this in  the proof. We shall proceed by induction on the sum of the lengths of $\D_1$ and $\D_2$. For the sum equal to two, the proposition obviously holds. We shall fix two segments as in the proposition, with the sum $\geq 3$, and suppose that the proposition holds for strictly smaller sums.
Suppose that the proposition does not hold for these $
\D_1$ and $\D_2$. From
$$
m^*_{bottom}(\d(\D_1)\t\d(\D_2))=\d(^-\D_1)\t\d(\D_2)\o b(\D_1)
+
\d(\D_1)\t\d(^-\D_2)\o b(\D_2)
$$
and the inductive assumption we conclude that the induced representation is of length 2. Denote the irreducible sub quotients by $\pi_1$ and $\pi_2$. After possible renumeration, we have
\begin{equation}
\label{jme1}
m^*_{bottom}(\pi_1)=\d(^-\D_1)\t\d(\D_2)\o b(\D_1),
\end{equation}
\begin{equation}
\label{jme2}
m^*_{bottom}(\pi_2)=\d(\D_1)\t\d(^-\D_2)\o b(\D_2).
\end{equation}
From the other side, in the Jacquet module of the induced representation we have an irreducible subquotient of the form
\begin{equation}
\label{sq}
\pi\o b(\D_1)\t b(\D_2).
\end{equation}
The assumption on $\D_1$ and $\D_2$ implies that the representation on the right hand side of the tensor product is irreducible. Denote the unique irreducible sub quotient of $\d(\D_1)\t\d(\D_2)$ which has \eqref{sq} in its Jacquet module by $\tau$. Then $\tau\leq \pi_1$ or $\tau\leq \pi_2$. Since $\tau$ has terms of the form $\dots\o b(\D_1)$ and $\dots\o b(\D_2)$ in its Jacquet module, we get contradiction with formula \eqref{jme1} and \eqref{jme2}.
 This contradiction completes the proof.
\end{proof}

We know that the representation $\rho\t\nu_\rho\rho$ reduces. One directly sees that this is a multiplicity one representation of length two. One irreducible subquotient is $\d([\rho,\nu_\rho\rho])$. Denote the other irreducible subquotient by 
$$
\mathfrak z([\rho,\nu_\rho\rho]).
$$
 
 For the classification of square integrable representations we shall need the following simple
\begin{lemma}
\label{small}
The representation $\rho\t\mathfrak z([\rho,\nu_\rho\rho])$ is irreducible\footnote{This follows directly applying the involution of Zelevinsky type which preserves the irreducibility (which is proved by A.-M . Aubert, and by P. Schneider and U. Stuhler).
Nevertheless,  we prefer not to use this very powerful result to prove this simple lemma.}.
\end{lemma}

\begin{proof}
One directly computes
$$
m^*(\rho\t\mathfrak z([\rho,\nu_\rho\rho]))=1\o \rho\t\mathfrak z([\rho,\nu_\rho\rho])
\hskip80mm
$$
\begin{equation}
\label{line1}
+
\rho\o\mathfrak z([\rho,\nu_\rho\rho]) 
+
\rho\o\rho\t\nu_\rho\rho
\hskip20mm
\end{equation}
\begin{equation}
\label{line2}
\hskip20mm
+\mathfrak z([\rho,\nu_\rho\rho])\o\rho+
\rho\t\rho\o\nu_\rho\rho
\end{equation}
$$
\hskip60mm
+
\rho\t\mathfrak z([\rho,\nu_\rho\rho])\o1.
$$
Suppose that $\rho\t\mathfrak z([\rho,\nu_\rho\rho])$ is reducible. Then from \eqref{line2} we see that there must be a subquotient $\pi$ of $\rho\t\mathfrak z([\rho,\nu_\rho\rho])$ which has 
$\mathfrak z([\rho,\nu_\rho\rho])\o\rho$ in its Jacquet module, and this is the whole Jacquet module  with respect to the corresponding parabolic subgroup which gives this Jacquet module (since the Jacquet module of $\rho\t\mathfrak z([\rho,\nu_\rho\rho])$ for that parabolic subgroup has length two). Now using the transitivity of Jacquet modules,  we get from \eqref{line1} that in the Jacquet module $\pi$ must be 
$\rho\o\d([\rho,\nu_\rho\rho])$, and this is the whole Jacquet module for the corresponding parabolic subgroup which gives this Jacquet module. From this, the Frobenius reciprocity implies
$$
\pi\h \rho\t\d([\rho,\nu_\rho\rho]).
$$
Now Lemma \eqref{subset} implies $\pi\cong \rho\t\d([\rho,\nu_\rho\rho])$. This is a contradiction, since the length of the minimal non-zero Jacquet module of $\pi$ is one, while the length  of the minimal non-zero Jacquet module of $\rho\t\d([\rho,\nu_\rho\rho])$ is three.
\end{proof}

\section{On Jacquet modules}

\begin{lemma}
\label{t1}
Let $\pi$ be an irreducible representation which has in its Jacquet module an irreducible subquotient
$$
\a_1\o\dots\o\a_n\o(\nu_\rho^{i+1}\rho\o\nu_\rho^i\rho\o\nu_\rho^i\rho)\o \b_1\o\dots\o\b_m,
$$
where $\rho$ is a cuspidal representation of some $GL(j,A)$.
 Then it has also
$$
\a_1\o\dots\o\a_n\o(\nu_\rho^{i}\rho\o\nu_\rho^{i+1}\rho\o\nu_\rho^i\rho)\o \b_1\o\dots\o\b_m
$$
in its Jacquet module.
\end{lemma}

\begin{proof} By the transitivity of the Jacquet modules, we can reduce this question to a question of irreducible representation $\pi'$ of $GL(3j,A)$ 
which has in its Jacquet module $\rho^{i+1}\rho\o\nu_\rho^i\rho\o\nu_\rho^i\rho$. Further, changing notation, we can take that $i=0$.

First observation is that $\pi'$ is a subquotient of  $\nu_\rho\rho\t\rho\t\rho$. 
Using the structure of the Hopf algebra, one directly sees that the multiplicity of  $\nu_\rho\rho\o\rho\o\rho$ in the Jacquet module of the above representation is 2. By Lemma \ref{subset}, $\rho\t\d([\rho,\nu_\rho\rho])$ is irreducible (which implies $\rho\t\d([\rho,\nu_\rho\rho])\cong
 \d([\rho,\nu_\rho\rho])\t\rho$).  Further, this representation 
is also subquotient of  $\nu_\rho\rho\t\rho\t\rho$. The multiplicity of $\nu_\rho\rho\o\rho\o\rho$ in the Jacquet module of $\rho\t\d([\rho,\nu_\rho\rho])$ is two. This implies $\pi'\cong \rho\t\d([\rho,\nu_\rho\rho])$. Now 
 the Frobenius reciprocity implies that 
$
\rho\o\nu_\rho\rho\o\rho
$
must be in the Jacquet module of $\pi'$. This implies
 the lemma.
\end{proof}

\begin{lemma}
\label{t2}
Let $\pi$ be an irreducible representation which has in its Jacquet module an irreducible subquotient
$$
\a_1\o\dots\o\a_n\o(\nu_\rho^{i+1}\rho\o\nu_\rho^i\rho\o\nu_\rho^{i+1}\rho)\o \b_1\o\dots\o\b_m,
$$
where $\rho$ is a cuspidal representation of some $GL(j,A)$.
 Then it has 
$$
\a_1\o\dots\o\a_n\o(\nu_\rho^{i+1}\rho\o\nu_\rho^{i+1}\rho\o\nu_\rho^i\rho)\o \b_1\o\dots\o\b_m
$$
or
$$
\a_1\o\dots\o\a_n\o(\nu_\rho^{i}\rho\o\nu_\rho^{i+1}\rho\o\nu_\rho^{i+1}\rho )\o\b_1\o\dots\o\b_m
$$
in its Jacquet module.

\end{lemma}

\begin{proof}
Again by the transitivity of the  Jacquet modules, we can reduce this question to a question about irreducible representation $\pi'$ of $GL(3j,A)$ 
which has in its Jacquet module $\rho^{i+1}\rho\o\nu_\rho^i\rho\o\nu_\rho^{i+1}\rho$. Further, changing notation, we can take again that $i=0$.

From Lemma \ref{small} we know that
 $\mathfrak z([\rho,\nu_\rho\rho])\t
\nu_\rho\rho$ is an irreducible representation.
This directly implies that
$\rho\t\nu_\rho\rho\t\nu_\rho\rho$ is a representation of length two. Observe that $\pi'$ must be one of these two irreducible subquotients.
One writes explicitly the composition series of minimal non-zero Jacquet modules of these two irreducible subquotients. Exponents of the subquotients of the representation $\d([\rho,\nu_\rho\rho])\t
\nu_\rho\rho$  are
$$
(1,0,1), (1,1,0), (1,1,0),
$$
and for the other one are
$$
(0,1,1), (0,1,1), (1,0,1).
$$
This implies that there must be also exponents $(1,1,0)$ or $(0,1,1)$. From this follows directly the lemma.

\end{proof}

\begin{remark} Let $\pi$ be an irreducible representation of some $GL(l,A)$ and let 
$$
\rho_1\o
\dots\o \rho_{i-1}\o (\rho_{i}\o \rho_{i+1})\o\rho_{i+2}\o \dots\o\rho_k
$$
be an irreducible cuspidal subquotient of some  standard Jacquet modules of $\pi$. Suppose  
that the representation
$$
\rho_i
\t\rho_{i+1}
$$
 is irreducible. Then the only irreducible representation which can have in its Jacquet module $\rho_i
\o\rho_{i+1}$ is $\rho_i
\t\rho_{i+1}$. Observe that this representations has also $\rho_{i+1}
\o\rho_{i}$ in its Jacquet module. Now from the transitivity of Jacquet modules directly follows that
$$
\rho_1\o
\dots\o \rho_{i-1}\o (\rho_{i+1}\o \rho_i)\o\rho_{i+2}\o \dots\o\rho_k
$$
is also an irreducible cuspidal subquotient of the same   standard Jacquet modules of $\pi$.

\end{remark}

\section{Square integrable representations}
\label{square}

Let $\pi$ be an irreducible representation of some $GL(n,A)$ and let
$$
\nu_{\rho_1}^{a_1}\rho_1\o\nu_{\rho_2}^{a_2}\rho_2\o \dots\o\nu_{\rho_k}^{a_k}\rho_k
$$
be an irreducible cuspidal subquotient of some  standard Jacquet modules of $\pi$, where we assume that $
\rho_i$ are unitarizable representations of $GL(n_i,A)$, and $a_i\in\mathbb R$ (clearly, $n_1+\dots+n_k=n)$\footnote{In this case, this Jacquet module is a minimal non-zero Jacquet module of $\pi$, and  all the other irreducible subquotients of this Jacquet module are cuspidal. Conversely, if we take a minimal non-zero Jacquet module of $\pi$, the all the irreducible subquotients of this Jacquet module are cuspidal.}.
Now the square integrability criterion of Casselman (Theorem 6.5.1 of \cite{Ca}) says that $\pi$ is square integrable if and only if for all irreducible subquotients as above holds 
\begin{equation}
\label{=}
\sum_{i=1}^k n_i s_{\rho_i} a_i=0,
\end{equation}
 and further if we have
\begin{equation}
\label{>}
\sum_{i=1}^j n_i  s_{\rho_i} a_i>0; \ \ 1\leq j \leq k-1.
\end{equation}
Observe that the formula \eqref{m^*} implies that representations $
\d(
\D)$ are unitarizable in the case that $\D^+=\D$.

A direct consequence of the above square integrability criterion and the above remark is the following

\begin{corollary}  Let $\pi$ be an irreducible square integrable representation of some $GL(l,A)$ and let 
$$
\rho_1\o\rho_2\o \dots\o\rho_k
$$
be an irreducible cuspidal subquotient of some  standard Jacquet modules of $\pi$.
Then the set $\{\rho_1,\rho_2, \dots,\rho_k\}$ is a segment of irreducible cuspidal representations\footnote{In the moment, we do not claim that there is no repetitions among representations $\rho_1,\rho_2, \dots,\rho_k$ (this will be proved  later)}.
\end{corollary}

\begin{proof} We present here a slightly different proof. Suppose that $\{\rho_1,\rho_2, \dots,\rho_k\}$ is not a segment of irreducible cuspidal representations. Then we can write this set as a disjoint union of two non-empty sets $X'$ and $X''$ such that
\begin{equation}
\label{commute}
\text{$\rho'\t\rho''$ is irreducible for all $\rho'\in X'$ and $\rho''\in X''$.}
\end{equation}
Take any irreducible quotient $\s$ of the Jacquet module  which has the representation  $\rho_1\o\rho_2\o \dots\o\rho_k$ for a subquotient. Then $\s\cong \rho_1'\o\rho_2'\o \dots\o\rho_k'$ for some irreducible cuspidal representations $\rho_i'$ of groups $G_l$'s, where $\rho_1',\rho_2' \dots,\rho_k'$ is some permutation of $\rho_1,\rho_2, \dots,\rho_k$ (see \cite{Ca} or \cite{Z}). The Frobenius reciprocity implies
$$
\pi \h \rho_1'\t\rho_2'\t \dots\t\rho_k'.
$$
Let $\a_1,\dots,\a_u$ be the subsequence of $\rho_1',\rho_2', \dots,\rho_k'$ consisting of all the representations which belong to $X'$. Analogously, let $\b_1,\dots,\b_v$ be the subsequence of $\rho_1',\rho_2', \dots,\rho_k'$ consisting of all the representations which belong to $X''$. Clearly, $1\leq u,v<k$ (and $u+v=k$). Now \eqref{commute} implies
$$
\pi\h \a_1\t\dots\t\a_u\t\b_1\t\dots\t\b_v\ \text{and} \ \pi\h\b_1\t\dots\t\b_v\t \a_1\t\dots\t\a_u.
$$
Let $\a_i\cong \nu_{\a_i^{(u)}}^{x_i}\a_i^{(u)}$, where $x_i\in\R$ and $\a_i^{(u)}$ is a unitarizable representation of $G_{c_i}$. Analogously, let $\b_j\cong \nu_{\b_j^{(u)}}^{y_j}\b_j^{(u)}$, where $y_j\in\R$ and $\b_j^{(u)}$ is a unitarizable representation of $G_{d_j}$. Then the above two embedding, the Frobenius reciprocity and the Casselmen square integrability criterion  imply
$$
\sum_{i=1}^u c_i  s_{\a_i^{(u)}}  x_i>0 \text{\ and \ } \sum_{j=1}^vd_j  s_{\b_j^{(u)}} y_j>0.
$$
 This implies $\sum_{i=1}^u c_i s_{\a_i^{(u)}} x_i+ \sum_{j=1}^vd_j s_{\b_j^{(u)}} y_j>0$, which contradicts to \eqref{=} of the Casselmen square integrability criterion. This contradiction completes the proof of the corollary.
\end{proof}

Now we can prove the following

\begin{proposition}
\label{si}
 Let $\s$ be an irreducible square integrable representation of some $GL(l,A)$. Then there exist a segment $\D$ of irreducible cuspidal representations such that
$$
\s=\d(\D)
$$
and $\D^+=\D$.
\end{proposition}

\begin{proof} The proof bellow is simple modification of \cite{J} (mainly of Lemma 2.2.2 there).
  By the previous corollary, there exists a unitarizable irreducible cuspidal representation $\rho$ and $0
\leq \a <1$ such that the cuspidal support of $\s$ is contained in $
\{\nu_\rho^{\a+z}\rho;z
\in 
\Z\}$. We now consider all  the irreducible cuspidal subquotients 
$$
\nu_{\rho}^{a_1}\rho\o\nu_{\rho}^{a_2}\rho\o \dots\o\nu_{\rho}^{a_k}\rho
$$
of the minimal standard non-trivial Jacquet module of $\s$. Among them, fix an irreducible sub quotient such that $(a_1,a_2,\dots,a_k)$  is minimal with respect to the lexicographical order on $\mathbb R^k$.  We know from the square integrability criterion that 
 $$
\sum_{i=1}^k a_i=0
$$
 and that all the following inequalities
$$
\sum_{i=1}^j a_i>0; \ \ 1\leq j \leq k-1,
$$
hold (observe that all the $s_{\rho_i}$ from the Casselman square integrability criterion are the same (and positive), so we can divide the relations in the criterion by this constant - we shall use this in the sequel).
Denote $c_1=a_1$.

Let $ 1\leq \ell_1
\leq k$ be the minimal index such that $a_{\ell_1+1}\geq a_1$, if such $\ell_1$ exists. If there is no $\ell_1$ such that $a_{\ell_1+1}\geq a_1$, we take $\ell_1$ to be $k$.

Suppose that $\ell_1>1$. Then the minimality and the above remark imply $a_2=a_1-1$.

Suppose $\ell_1>2$. Then  cannot have $a_3<a_1-2$ (since if this would be the case, using the above remark we would get a contradiction to the minimality). Suppose that $a_3>a_2-2$. Then we must have $a_3=a_1-1$. Now  Lemma \ref{t1} implies that there exists strictly smaller term (with respect to the lexicographic ordering) in the Jacquet module of $\s$. This contradiction implies that $a_3=a_1-2$.

Now we shall prove that in general holds
$$
a_i=a_1-i+1, \ \  1\leq i\leq \ell_1.
$$
We prove this by induction. We have seen that this hold if $i\leq 3$ and $i\leq \ell_1$. Suppose $3\leq l <\ell_1$, and that the above claim holds for $i\leq l$. 

Suppose $a_{l+1}\ne a_1-l$. First, we cannot have  $a_{l+1}< a_1-l$ (since if this would be the case, using the above remark we would get a contradiction to the minimality again). Suppose  $a_{l+1}= a_1-l+1$. Now applying Lemma \ref{t1}, we get a contradiction to the minimality. 

Thus,
$$
 a_1-l+2\leq a_{l+1}\leq a_1-1.
$$
Therefore,
$$
a_{l+1}=a_1-j
$$
for some $j$ satisfying
$$
1\leq j\leq l-2.
$$

Now applying the above remark several times, we would get that in the Jacquet module of $\s$ must be a term of the form
$$
\nu_\rho^{a_1}\rho\o \nu_\rho^{a_1-1}\rho
\o  \dots\o\nu_\rho^{a_1-j+2}\rho\o
\nu_\rho^{a_1-j+1}\rho\o
(\nu_\rho^{a_1-j}\rho
\o
\nu_\rho^{a_1-j-1}\rho\o \nu_\rho^{a_1-j}\rho)\o\dots.
$$
Lemma \ref{t2} implies that we must have at least one of the additional two representations listed in that lemma. The minimality implies that we cannot have the  last representation listed there.  Therefore, we must have a representation of the form
$$
\nu_\rho^{a_1}\rho\o \nu_\rho^{a_1-1}\rho
\o  \dots\o
\nu_\rho^{a_1-j+2}\rho\o
\nu_\rho^{a_1-j+1}\rho\o
(\nu_\rho^{a_1-j}\rho
\o
\nu_\rho^{a_1-j}\rho\o \nu_\rho^{a_1-j-1}\rho)\o\dots
$$
$$
=\nu_\rho^{a_1}\rho\o \nu_\rho^{a_1-1}\rho
\o  \dots\o
\nu_\rho^{a_1-j+2}\rho\o
(\nu_\rho^{a_1-j+1}\rho\o
\nu_\rho^{a_1-j}\rho
\o
\nu_\rho^{a_1-j}\rho)\o \nu_\rho^{a_1-j-1}\rho\o\dots.
$$
Now Lemma \ref{t1} implies that we must have a term of the form
$$
\nu_\rho^{a_1}\rho\o \nu_\rho^{a_1-1}\rho
\o  \dots\o \nu_\rho^{a_1-j+2}\rho\o
(\nu_\rho^{a_1-j}\rho\o
\nu_\rho^{a_1-j+1}\rho
\o
\nu_\rho^{a_1-j}\rho)\o \nu_\rho^{a_1-j-1}\rho\o\dots
$$
in the Jacquet module of $\s$.
This contradicts to the minimality. Therefore, we have completed the proof of our claim.

Continuing this procedure, we get that the minimal element is of the form
$$
(\nu_\rho^{c_1}\rho\o \nu_\rho^{c_1-1}\rho
\o  \dots\o
\nu_\rho^{d_1}\rho)\o
(\nu_\rho^{c_2}\rho
\o
\nu_\rho^{c_2-1}\rho\o \dots\o \nu_\rho^{d_2}\rho)\o\dots
\o
(\nu_\rho^{c_s}\rho\o \nu_\rho^{c_s-1}\rho
\o  \dots\o
\nu_\rho^{d_s}\rho),
$$
where
$$
c_1\leq c_2\leq \dots\leq c_s.
$$

This implies 
$$
\s\leq
(\nu_\rho^{c_1}\rho\t\nu_\rho^{c_1-1}\rho
\t  \dots\t
\nu_\rho^{d_1}\rho)\t
(\nu_\rho^{c_2}\rho
\t
\nu_\rho^{c_2-1}\rho\t \dots\t \nu_\rho^{d_2}\rho)\t\dots
\t(
\nu_\rho^{c_s}\rho\t \nu_\rho^{c_s-1}\rho
\t  \dots\t
\nu_\rho^{d_s}\rho).
$$
Now one can easily show (using the Frobenius reciprocity and the induction in stages) that there exist  irreducible subquotients $\s_i$ of 
$$
\nu_\rho^{c_i}\rho\t\nu_\rho^{c_i-1}\rho
\t  \dots\t
\nu_\rho^{d_i}\rho
$$
such that
$$
\s\h \s_1\t \dots\t\s_s.
$$
Suppose that some $\s_i\ne \d([ \nu_\rho^{d_i}\rho,  \nu_\rho^{c_i}\rho ])$. Chose the lowest  index $i$ for which this holds. Then 
$$
\nu_\rho^{c_i}\rho\o\nu_\rho^{c_i-1}\rho
\o  \dots\o
\nu_\rho^{d_i}\rho
$$
in not in the Jacquet module of $\s_i$, and all the  irreducible cuspidal subquotients of $\s$ are strictly smaller lexicographically. This would directly produce an irreducible subquotient strictly smaller then the minimal one. This contradiction implies that  $\s_i\cong  \d([ \nu_\rho^{d_i}\rho,  \nu_\rho^{c_i}\rho ])$ for all indexes $i$.

Thus,
\begin{equation}
\label{emb}
\s\h  \d([ \nu_\rho^{d_1}\rho,  \nu_\rho^{c_1}\rho ])\t
\dots\t
\d([ \nu_\rho^{d_s}\rho,  \nu_\rho^{c_s}\rho ]).
\end{equation}

If $s=1$, then the proposition obviously hold. Suppose $s\geq 2$. The square integrability criterion implies 
$$
c_1+d_1>0.
$$
The (single) equality relation in the square integrability criterion implies that we can not have 
$$
c_i+d_i>0
$$
for 
all indexes $i$. Choose the lowest index $i$ such that this is not the case.
Then we know that holds
$$
c_1\leq \dots \leq c_i,
$$
$$
-c_1<d_1,\dots, -c_{i-1}<d_{i-1},
$$
and
$$
d_i\leq -c_i.
$$
This implies that for $j<i$ hold
$$
   d_i \leq -c_i\leq -c_j< d_j,
$$
which implies $d_i<d_j$ and $c_j\leq c_i$. Therefore, the segment $[ \nu_\rho^{d_i}\rho,  \nu_\rho^{c_i}\rho ]$ contains all the segments $[ \nu_\rho^{d_j}\rho,  \nu_\rho^{c_j}\rho ]$ with $j<i$.
 Now Lemma \ref{subset} and the  relation \eqref{emb} imply
$$
\s\h  \d([ \nu_\rho^{i}\rho,  \nu_\rho^{i}\rho ])\t
\dots.
$$
Using the Frobenius reciprocity and applying the square integrability criterion, we get 
$
c_i+d_i>0.
$
This obviously contradicts our assumption 
$
c_i+d_i\leq 0.
$
This contradiction completes the proof.
\end{proof}

Now Propositions \ref{link}, \ref{si} and  \ref{irr-uni-seg}\footnote{Instead of Proposition \ref{irr-uni-seg}, we can use 
the Harish-Chandra commuting algebra theorem.
This theorem and Proposition \ref{link} imply that all the generators of the commuting algebra of a representation parabolically induced by an irreducible square integrable representation are scalar operators. This implies the irreducibility of the induced representation.} imply the following

\begin{corollary} 
The tempered induction for $GL(n,A)$ is irreducible. In other words, if $\tau_1,\dots,\tau_k$ are irreducible tempered representations of general linear groups over $A$, the $\tau_1\t\dots\t\tau_k$ is irreducible (and tempered).
\qed
\end{corollary}

\section{Appendix - the case of two linked segments}

Considerable  part of  \cite{T-Crelle} is devoted to the determining of the composition series of $\d(\D_1)\t\d(\D_2)$ for linked segments $\D_1$ and $\D_2$. This information  is obtained there using the properties of the Langlands classification. This can be obtained also  using the Jacquet modules (it is a little bit  more elementary). Since we have presented proofs of two key facts of the representation theory of groups $G_n$ based on Jacquet modules, it is natural to have also the proof of this important fact based  on Jacquet modules.
We  present such proof below.

Let $\D_1$ and $\D_2$ be linked segments. We can write them as
$$
\D_i=[ \nu_\rho^{n_i}\rho,  \nu_\rho^{m_i}\rho ], \qquad i=1,2,
$$
where $n_i,m_i\in \R$.
After possible changing of indexes, we shall assume $n_1<n_2$. We shall use the interpretation of the Langlands classification in term of finite multisets of segments  (see \cite{T-Crelle}).  For our purpose, we need only to know that $L(\D_1,\D_2)$ denotes the unique irreducible quotient of $\d(\D_2)\t\d(\D_1)$.

\begin{lemma}
\label{l1}
Suppose
$$
m_1+1=n_2.
$$
 Then in the Grothendieck group of the category of representations of finite length we have
$$
\d(\D_1)\t\d(\D_2)=L(\D_1,\D_2)+\d(\D_1\cup\D_2).
$$
\end{lemma}

\begin{proof} Observe that 
$$
\d([\rho^{n_1},\nu_\rho^{m_2}\rho])\h\nu_\rho^{m_2}\rho\t\nu_\rho^{m_2-1}\rho\t\dots \t\nu_\rho^{n_1}\rho,
$$
$$
\d([\rho^{n_2},\nu_\rho^{m_2}\rho]) \t
\d([\rho^{n_1},\nu_\rho^{m_1}\rho])\h
(\nu_\rho^{m_2}\rho\t\nu_\rho^{m_2-1}\rho\t\dots \t\nu_\rho^{n_2}\rho) \t
(\nu_\rho^{m_1}\rho\t\nu_\rho^{m_1-1}\rho\t\dots \t\nu_\rho^{n_1}\rho)
$$
(we use here that $m_1+1=n_2$). Since the representations on the right hand sides are isomorphic, and they have the unique irreducible subrepresentations, we get that $\d(\D_1\cup\D_2)\h\d(\D_2)\t\d(\D_1)$. Thus, $\d(\D_1\cup\D_2)$ is a subquotient of $\d(\D_1)\t\d(\D_2)$ (recall that $R$ is commutative). Therefore, $\d(\D_1\cup\D_2)$ and $L(\D_1,\D_2)$ are sub quotients of $\d(\D_1)\t\d(\D_2)$. Obviously, $\d(\D_1\cup\D_2)$ and $L(\D_1,\D_2)$ are not isomorphic (since their Langlands parameters are different). 

The lemma will be proved if we show that $\d(\D_1)\t\d(\D_2)$ has length at most two.
We shall prove the this by induction of the sum of cardinalities of $\D_1$ and $\D_2$. If the sum is two, then we know this (since the minimal non-zero (standard) Jacquet module is of length two). Suppose that the sum of cardinalities of $\D_1$ and $\D_2$ is $\geq 3$, and that the claim holds for lower sums of cardinalities. Write
$$
m_{bottom}^*(\d(\D_1)\t\d(\D_2))=\d(\D_1)\t\d(^-\D_2)\o\nu_\rho^{n_2}\rho+
\d(^-\D_1)\t\d(\D_2)\o\nu_\rho^{n_1}\rho.
$$
The first representation on the right hand side is irreducible by Proposition \ref{link}. 

If $n_1=m_1$, then also the second one is obviously  irreducible. This implies that the multiplicity is at most two in the case $n_1=m_1$.

It remains to consider the case If $n_1<m_1$. Now the inductive assumption implies that the second representation on the right hand side has length two. Therefore, the length of the above Jacquet module is three. Let $\pi$ be an irreducible sub quotient of $\d(\D_1)\t\d(\D_2)$ which has the first representation on the right hand side in its Jacquet module. Then  $\pi$ has an irreducible representation of the form
$$
\dots \o \nu_\rho^{n_1}\rho\o \nu_\rho^{n_2}\rho
$$
in its Jacquet module. Observe that $n_1+1\leq m_1=n_2-1$, which gives $n_1+2\leq n_2$. Therefore, 
$\nu_\rho^{n_1}\rho\o \nu_\rho^{n_2}\rho$ is irreducible. Now Remark 3.3 implies that $\pi$ has also an irreducible representation of the form
$$
\dots \o \nu_\rho^{n_2}\rho\o \nu_\rho^{n_1}\rho
$$
in its Jacquet module. By the transitivity of Jacquet modules, this term  must come from some term of $m_{bottom}^*(\d(\D_1)\t\d(\D_2))$, and it can not obviously come from the first representation on the right hand side. Therefore, we must have in the Jacquet module of $\pi$ also at least one sub quotient of the second representation on the right hand side. This implies that the length of $\d(\D_1)\t\d(\D_2)$ is at most two. The proof of the lemma is now complete.
\end{proof}

We continue with previous linked segments $\D_1$ and $\D_2$. Decompose $\D_1$ into two segments
$$
\D_{1,l}=\d([\rho^{n_1},\nu_\rho^{n_2-1}\rho]), \qquad \D_{1,r}=\d([\rho^{n_2},
\nu_\rho^{m_1}\rho]),
$$
where we take the second segment to be empty set if $n_2=m_1+1$. Then $\D_{1,r}=\D_1\cap\D_2$ and $\D_{1,l}\cup\D_{2}=\D_1\cup\D_2$ ($\D_{1,l}$ and $\D_{2} $ are disjoint). Now the last lemma implies
\begin{equation}
\label{in1}
\d(\D_{1})\t\d(\D_{2})\leq \d(\D_{1,l})\t\d(\D_{1,r})\t \d(\D_{2}),
\end{equation}
\begin{equation}
\label{in2}
\d(\D_{1}\cup\D_{2})\t\d(\D_{1}\cap\D_{2})\leq \d(\D_{1,l})\t\d(\D_{1,r})\t \d(\D_{2}).
\end{equation}

\begin{lemma} 
\label{l2}
With above notation we have
\begin{enumerate}
\item \qquad
$
\d(\D_{1})\t\d(\D_{2})+
\d(\D_{1}\cup\D_{2})\t\d(\D_{1}\cap\D_{2})
\not\leq \d(\D_{1,l})\t\d(\D_{1,r})\t \d(\D_{2}).
$

\
\item
\hskip30mm
$
\d(\D_{1}\cup\D_{2})\t\d(\D_{1}\cap\D_{2})\leq
\d(\D_{1})\t\d(\D_{2}).
$

\end{enumerate}
\end{lemma}

\begin{proof} We prove (1) by induction with respect to $n_2-n_1=k$ (clearly, $k\geq 1$). 

Suppose that  the inequality does not hold for the difference $k$ (i.e. that the inequality holds).
Then obviously
$$
m^*_{-\o\nu^{n_1}_\rho\rho}(\d(\D_{1})\t\d(\D_{2})+\d(\D_{1}\cap\D_{2})\t\d(\D_{1}\cup\D_{2}))
\leq m^*_{-\o\nu^{n_1}_\rho\rho}(\d(\D_{1,l})\t\d(\D_{1,r})\t \d(\D_{2})).
$$
Now the formula \eqref{m^*} implies
$$
\d(^-\D_{1})\t\d(\D_{2})\o\nu^{n_1}_\rho\rho+
\d(\D_{1}\cap\D_{2})\t\d(^-(\D_{1}\cup\D_{2}))\o\nu^{n_1}_\rho\rho
\leq \d(^-\D_{1,l})\t\d(\D_{1,r})\t \d(\D_{2})\o\nu^{n_1}_\rho\rho.
$$
This obviously implies 
\begin{equation}
\label{A}
\d(^-\D_{1})\t\d(\D_{2})+
\d(\D_{1}\cap\D_{2})\t\d(^-(\D_{1}\cup\D_{2}))
\leq \d(^-\D_{1,l})\t\d(\D_{1,r})\t \d(\D_{2}).
\end{equation}

Suppose now $k=1$ and that the inequality does not hold in this case. Then the above inequality becomes in this case
$$
\d(^-\D_{1})\t\d(\D_{2})+
\d(\D_{1}\cap\D_{2})\t\d(^-(\D_{1}\cup\D_{2}))
\leq \d(^-\D_{1})\t \d(\D_{2}),
$$
which obviously can not hold. Therefore, the inequality holds for $k=1$.

Fix now $k\geq 2$, and assume that the inequality does not hold for the difference $k-1$. Observe that then $^-\D_1$ and $\D_2$ are linked, and that for them the inequality does not hold by the inductive assumption. Suppose that  the inequality  holds for the difference $k$. 
Now \eqref{A} and the fact that $^-(\D_{1}\cap\D_{2})=(^-\D_{1})\cap\D_{2}$ and $^-(\D_{1}\cup\D_{2})=(^-\D_{1})\cup\D_{2}$ imply
$$
\d(^-\D_{1})\t\d(\D_{2})+
\d(^-(\D_{1})\cap\D_{2})\t\d(^-(\D_{1})\cup\D_{2})
\leq \d(^-\D_{1,l})\t\d(\D_{1,r})\t \d(\D_{2}).
$$
Observe that $^-\D_1=(^-\D_{1,l})\cup\D_{1,r}$ is the decomposition of $^-\D_1$, corresponding the such a decomposition for $\D_1$.
This  contradicts  the inductive assumption, and it completes the proof of (1).

 Now (2) follows directly from \eqref{in1}, \eqref{in2} and (1), using the fact that 
$\d(\D_{1}\cup\D_{2})\t\d(\D_{1}\cap\D_{2})$ is irreducible (which follows from Lemma \ref{subset}).
\end{proof}

\begin{proposition}
Let $\D_1$ and $\Delta_2$ be two linked segments. Then in the Grothendieck group of the category of representations of finite length we have
$$
\d(\D_1)\t\d(\D_2)=L(\D_1,\D_2)+\d(\D_1\cup\D_2)\t\d(\D_1\cap\D_2).
$$
\end{proposition}

\begin{proof} Observe that  by Lemma \ref{l1}, we know that the proposition holds in the case when $\D_1\cap\D_2=\emptyset$. Therefore, in the rest of the proof we need to concentrate to the case when $\D_1\cap\D_2\ne\emptyset$. This implies that $\D_1$ has at least two elements (and also $\D_2$).

 From the previous lemma, we know that  $\d(\D_1\cup\D_2)\t\d(\D_1\cap\D_2)$ is a sub quotient of $\d(\D_1)\t\d(\D_2)$. We know also that $L(\D_1,\D_2)$ is a subquotient. Therefore, for the proof of the proposition, it is enough to prove that the length of $\d(\D_1)\t\d(\D_2)$ is (at most) two.

The proof goes by induction with respect to card$(\D_1\cup\D_2)$. For the cardinality two and three, we have $\D_1\cap\D_2=\emptyset$, so the proposition holds in this case. Suppose that the cardinality of $\D_1\cup\D_2$ is $\geq 4$, and that the claim holds for strictly smaller cardinalities. Recall
\begin{equation}
\label{bottom}
m_{bottom}^*(\d(\D_1)\t\d(\D_2))=\d(\D_1)\t\d(^-\D_2)\o\nu_\rho^{n_2}\rho+
\d(^-\D_1)\t\d(\D_2)\o\nu_\rho^{n_1}\rho.
\end{equation}
Since the  subquotient $\d(\D_1\cup\D_2)\t\d(\D_1\cap\D_2)$ embeds into $\d(\D_1\cup\D_2)\t\nu_\rho^{m_1}\rho\t\dots \t \nu_\rho^{n_2}\rho$, Frobenius reciprocity implies that this sub quotient has in its Jacquet module a term of the form $\dots \o\nu_\rho^{n_2}\rho$. Now \eqref{bottom} implies that $\d(\D_1\cup\D_2)\t\d(\D_1\cap\D_2)$ must have in its Jacquet module an irreducible sub quotient of $\d(\D_1)\t\d(^-\D_2)\o\nu_\rho^{n_2}\rho$. We know from Lemma \ref{subset} that $\d(\D_1\cup\D_2)\t\d(\D_1\cap\D_2)\cong \d(\D_1\cap\D_2)\t\d(\D_1\cup\D_2)$. Now in  the same way as above, we conclude that $\d(\D_1\cup\D_2)\t\d(\D_1\cap\D_2)$ must have in its Jacquet module an irreducible sub quotient of 
$\d(^-\D_1)\t\d(\D_2)\o\nu_\rho^{n_1}\rho$.

We consider now two cases. The first is $n_1+1=n_2$. In this case,  the inductive assumption and Lemma \ref{subset} imply that \eqref{bottom} has length three.
Now the above analysis of  the Jacquet module of   $\d(\D_1\cup\D_2)\t\d(\D_1\cap\D_2)$ implies that $\d(\D_1)\t\d(\D_2)$ has the length (at most) two.

It remains to consider the case $n_1+2\leq n_2$. Observe that the irreducible subquotient $L(\D_1,\D_2)$ imbeds into  $\d(\D_1)\t\d(\D_2)$ (this is an elementary property of the Langlands classification). Now in the same was as above, we conclude that $L(\D_1,\D_2)$ must have in its Jacquet module an irreducible subquotient of $\d(\D_1)\t\d(^-\D_2)\o\nu_\rho^{n_2}\rho$. From \eqref{m^*} follows that $\d(\D_1)\h \d(^-\D_{1})\t\nu_\rho^{n_1}\rho$. Now $n_1+2\leq n_2$ and Lemma \ref{subset} imply $\nu_\rho^{n_1}\rho\t\d(\D_2)\cong \d(\D_2)\t\nu_\rho^{n_1}\rho$. 
From this follows
$$
L(\D_1,\D_2)\h \d(\D_1)\t\d(\D_2)\h\d(^-\D_{1})\t\nu_\rho^{n_1}\rho\t\d(\D_2)\cong \d(^-\D_{1})\t\d(\D_2)\t\nu_\rho^{n_1}\rho.
$$
Now in the same way as above, we conclude that $L(\D_1,\D_2)$ must have in its Jacquet module an irreducible subquotient of $\d(^-\D_1)\t\d(\D_2)\o\nu_\rho^{n_1}\rho$. Since by the inductive assumption the length of \eqref{bottom} in the case $n_1+2\leq n_2$  is four, the above observations about Jacquet modules of $\d(\D_1\cup\D_2)\t\d(\D_1\cap\D_2)$  and $L(\D_1,\D_2)$ imply that the length of $\d(\D_1)\t\d(\D_2)$ is two. This completes the proof of the proposition.
\end{proof}

Now using the factorization of the long intertwining operator from the Langlands classification, one gets from the above proposition the composition series of the generalized principal series (see Theorem 5.3 of \cite{T-Crelle}).

\end{document}